\documentclass{scrartcl}
\KOMAoptions{
  fontsize=10pt
  ,paper=a4
  ,parskip=half
  ,headings=normal
  ,titlepage=false
}

\usepackage[utf8]{inputenc}
\usepackage[T1]{fontenc}
\usepackage[sc,osf]{mathpazo}
\usepackage{fourier-orns}
\usepackage{enumerate}

\usepackage{array}
  \newcolumntype{d}[1]{>{\raggedleft\hspace{0pt}}p{#1}} % New column type for right aligned constant width column
  \newcolumntype{i}[1]{>{\raggedright\hspace{0pt}}p{#1}} % New column type for left aligned constant width column
  \newcolumntype{x}[1]{>{\centering\let\newline\\\arraybackslash\hspace{0pt}}p{#1}} % New column type for centered constant width column

\usepackage{graphicx}
\usepackage[usenames,dvipsnames]{color}
  \definecolor{alert}{rgb}{1,0,0}
  
\usepackage{xspace}

\usepackage{amsmath,amsthm,amssymb}
  \numberwithin{equation}{section}

\usepackage{hyperref}
  \hypersetup{
    colorlinks,
    pdfauthor={Francisco Torralbo},
    pdftitle={Parallel mean curvature surfaces in 4-manifolds - State of Art},
    pdfsubject={PYR-GENIL Research project},
    urlcolor=[rgb]{0.2,0.2,0.2},
    linkcolor=[rgb]{0.2,0.2,0.2},
    citecolor=[rgb]{0.2,0.2,0.2}
  }

  \newcommand\blfootnote[1]{%
  \begingroup
  \renewcommand\thefootnote{}\footnote{#1}%
  \addtocounter{footnote}{-1}%
  \endgroup
}

  \title{\Large{Parallel mean curvature surfaces in four-dimensional homogeneous spaces}}
  \author{\normalsize\textsc{José M.\ Manzano, Francisco Torralbo, and Joeri Van der Veken}}
  \date{}

  \newtheorem{theorem}{Theorem}[section]
  \newtheorem{proposition}{Proposition}[section]
  
  \newtheorem{lemma}{Lemma}[section]

  \theoremstyle{definition}
    \newtheorem{definition}{Definition}[section]

  \theoremstyle{remark}
    \newtheorem{remark}{Remark}[section]

\begin{document}

\maketitle

\begin{abstract}
\noindent\textsc{Abstract.} We survey different classification results for surfaces with parallel mean curvature immersed into some Riemannian homogeneous four-manifolds, including real and complex space forms and product spaces. We provide a common framework for this problem, with special attention to the existence of holomorphic quadratic differentials on such surfaces. The case of spheres with parallel mean curvature is also explained in detail, as well as the state-of-the-art advances in the general problem.\blfootnote{Support: Spanish MEC-Feder Research Project MTM2014-52368-P (first and second authors), EPSRC Grant No. EP/M024512/1 (first author); Belgian Interuniversity Attraction Pole P07/18 ``Dynamics, Geometry and Statistical Physics" (second and third authors); KU Leuven Research Fund project 3E160361 ``Lagrangian and calibrated submanifolds" (third author). The authors would like to thank Prof.\ K.\ Kenmotsu for his valuable comments. The second author would also like to thank the Geometry Section of the Department of Mathematics at KU Leuven for their kindness during his stay at KU Leuven.}\\[-5pt]


\noindent\textsc{MSC 2010.} Primary $53\text{C}42$; Secondary $53\text{C}30$.
\end{abstract}

\section{Introduction}

Surface Theory in three-dimensional manifolds is a classical topic in Differential Geometry. Although the most extensive investigation has been carried out in ambient three-manifolds with constant curvature, the so-called \emph{space forms}, there has been a growing interest in considering the broader family of homogeneous three-manifolds. Among the different geometrically distinguished families of surfaces, we will focus on those which have constant mean curvature (\textsc{cmc} in the sequel). When the codimension is bigger than one, a natural generalization of \textsc{cmc} surfaces are those whose mean curvature vector is not constant but parallel in the normal bundle. These surfaces are called parallel mean curvature surfaces (\textsc{pmc} from now on, see Definition~\ref{def:pmc}), and enjoy some of the properties of \textsc{cmc} surfaces in codimension one.

The aim of this work is to gather some results on the classification of \textsc{pmc} surfaces when the codimension is two and the ambient space is homogeneous, sketching a parallelism with \textsc{cmc} surfaces in homogeneous three-manifolds. The connection between these two theories principally comes from the fact that \textsc{cmc} surfaces in totally umbilical \textsc{cmc} hypersurfaces of a four-manifold become \textsc{pmc} surfaces (cf.\ Proposition~\ref{prop:cmc-totally-umbilical}). When the four-manifold is homogeneous typically such a hypersurface is also homogeneous. Nonetheless, there could be \textsc{pmc} surfaces not factoring through a hypersurface in this sense, as we will discuss below. The interested reader can refer to~\cite{DHM2009} and~\cite{MP2012} for an introduction to \textsc{cmc} surfaces in homogeneous three-manifolds. Another approach that covers both \textsc{cmc} and \textsc{pmc} surfaces as critical points of extended area functionals can be found in~\cite{Salavessa2010}.

In the seventies, Ferus~\cite{Ferus71} proved that an immersed \textsc{pmc} sphere in a space form is a round sphere (cf.\ Theorem~\ref{thm:pmc-spheres-space-forms}), and afterwards Chen~\cite{Chen73} and Yau~\cite{Yau74} classified \textsc{pmc} surfaces in space forms, showing that they are \textsc{cmc} surfaces in three-dimensional totally umbilical hypersurfaces (cf.\ Theorem~\ref{thm:classification-pmc-space-forms}). It is also important to mention the contribution of Hoffman~\cite{Hoffman1973}, who classified \textsc{pmc} surfaces of $\mathbb{R}^4$ and $\mathbb{S}^4$ in terms of analytical functions assuming their Gauss curvature does not change sign.

Almost thirty years later, Kenmotsu and Zhou~\cite{KZ00} undertook the classification of \textsc{pmc} surfaces in the complex space forms $\mathbb{C}\mathrm{P}^2$ and $\mathbb{C}\mathrm{H}^2$, based on a result of Ogata~\cite{Ogata95}. However, soon thereafter Hirakawa~\cite{Hirakawa2006}  pointed out a mistake in Ogata's equation, but gave a classification of the \textsc{pmc} spheres in $\mathbb{C}\mathrm{P}^2$ and $\mathbb{C}\mathrm{H}^2$. The mistake was corrected in~\cite{KO2015} but the classification was still incomplete. Finally, Kenmotsu has recently published a correction~\cite{Kenmotsu16} that closes the classification problem. The complete classification then follows from both~\cite{Hirakawa2006} and~\cite{Kenmotsu16}.

The classification of \textsc{pmc} surfaces in four-dimensional manifolds has also been treated in $\mathbb M^3(c) \times \mathbb{R}$, where $\mathbb M^n(c)$ denotes the $n$-dimensional space form of constant sectional curvature $c$. On the one hand, de Lira and Vit\'{o}rio~\cite{LV08} classified the \textsc{pmc} spheres (cf.\ Theorem~\ref{thm:pmc-spheres-space-formxR}). On the other hand, Alencar, do Carmo and Tribuzy~\cite{AdCT2010} proved reduction of codimension for \textsc{pmc} surfaces in $\mathbb M^n(c) \times \mathbb{R}$ (cf.\ Theorem~\ref{thm:classification-pmc-surfaces-space-formxR}) also classifying \textsc{pmc} spheres (cf.\ Theorem~\ref{thm:pmc-spheres-M4xR}). Mendon\c{c}a and Tojeiro~\cite{MT2014} improved Alencar, do Carmo and Tribuzy's result under some additional conditions (see Section~\ref{subsec:general-M4}).

A few years ago, the second author and Urbano~\cite{TU12} classified the \textsc{pmc} spheres in the product four-manifolds $\mathbb{S}^2 \times \mathbb{S}^2$ and $\mathbb{H}^2 \times \mathbb{H}^2$, as well as a large family of \textsc{pmc} surfaces that satisfy an extra condition on the extrinsic normal curvature (cf.\ Theorem~\ref{thm:clasification-pmc-M2xM2}). Fetcu and Rosenberg also tackled the problem in other ambient manifolds obtaining several partial results, namely, in $\mathbb{S}^3\times \mathbb{R}$ and $\mathbb{H}^3\times \mathbb{R}$~\cite{FR2012}, in $\mathbb M^n(c) \times \mathbb{R}$~\cite{FR2013}, in $\mathbb{C}\mathrm{P}^n\times \mathbb{R}$ and $\mathbb{C}\mathrm{H}^n\times \mathbb{R}$~\cite{FR2014} and also in Sasakian space forms~\cite{FH2014}, including the Heisenberg space of any odd dimension.

An interesting family where to study the classification problem for \textsc{pmc} surfaces in is that of the four-dimensional Thurston geometries, i.e., homogeneous four-manifolds whose isometry group acts transitively and effectively on them, and the stabilizer subgroup at each point is compact. Usually the isometry group is required to be maximal in the sense that it cannot be enlarged to another subgroup. Under these assumptions, there are 19 types of Thurston geometries in dimension 4, listed in Table \ref{tb:Thurston-4-geometries} below. We will emphasize the product geometries, which might be the first spaces where \textsc{pmc} surface should be understood:
\begin{itemize}
 \item The product spaces $\mathbb{M}^3(c)\times \mathbb{R}$ (see Sections~\ref{subsec:holomorphic-differentials-M3xR} and~\ref{subsec:pmc-surfaces-M3xR}). The classification of the \textsc{pmc} spheres was done by de Lira and Vit\'{o}rio~\cite{LV08}, but the general classification remains open.

  \item The product spaces $\mathbb M^2(c_1)\times\mathbb{M}^2(c_2)$. The classification of spheres is known when $c_1=c_2$ (see Sections~\ref{subsec:holomorphic-differentials-Riemannian-products} and~\ref{subsec:pmc-surfaces-Riemannian-products}), but the general case remains still open, although there are some partial results (see Section~\ref{subsec:holomorphic-differentials-S2xH2}).

  \item The product spaces $\mathrm{Nil}_3 \times \mathbb{R}$, $\widetilde{\mathrm{Sl}}_2(\mathbb{R})\times \mathbb{R}$ and $\mathrm{Sol}^3\times\mathbb{R}$ (the latter is included in the family $\mathrm{Sol}_{m,n}^4$ in Table~\ref{tb:Thurston-4-geometries}).
\end{itemize}
Dropping the condition on the maximality of the isometry group, a simply connected homogeneous four-dimensional product manifold is either of the form $\mathbb M^2(c_1)\times\mathbb{M}^2(c_2)$ or $G\times\mathbb{R}$, where $G$ is a Lie group endowed with a left-invariant metric (see~\cite{MP2012}). In the latter family, it is worth highligthing the family $\mathbb{E}(\kappa, \tau) \times \mathbb{R}$ where $\mathbb{E}(\kappa, \tau)$, $\kappa -4\tau^2 \neq 0$, denotes the two-parameter family of simply connected three-manifolds with isometry group of dimension four (see~\cite{VdV08}, \cite{DHM2009} and the references therein).

The existence of holomorphic quadratic differentials for \textsc{pmc} surfaces has been central in their classification. Note that \textsc{cmc} surfaces in $\mathbb{E}(\kappa,\tau)$-spaces admit a holomorphic quadratic differential called the Abresch-Rosenberg differential~\cite{AR05}. This fact plays a key role in the definition of holomorphic quadratic differentials for \textsc{pmc} surfaces in $\mathbb{H}^3\times\mathbb{R}$ and $\mathbb{S}^3\times\mathbb{R}$ (see Section~\ref{subsec:pmc-surfaces-M3xR}).
%, and it is expected to give some extra information for \textsc{pmc} surfaces in $\mathbb{E}(\kappa,\tau)\times\mathbb{R}$.

\begin{table}[htbp]
  \centering
  \begin{tabular}{p{.35\textwidth}cc p{.23\textwidth}}
    \hline %-------------------------------
    \textsc{Geometry} & \textsc{Isotropy}   &  $\mathrm{dim}(\mathrm{Iso})$ & \textsc{K{\"a}hler}\\
    \hline %-------------------------------
    $\mathbb{S}^4$, $\mathbb{R}^4$, $\mathbb{H}^4$ & $\mathrm{SO}_4$  &   10  &  No, except $\mathbb{R}^4$ \\
    $\mathbb{C}\mathrm{P}^2$, $\mathbb{C}\mathrm{H}^2$ & $\mathrm{U}_2$   &   9  &  Yes \\
    $\mathbb{S}^3\times \mathbb{R}$, $\mathbb{H}^3\times \mathbb{R}$ &     $\mathrm{SO}_3$ &  7 & No \\
    $\mathbb{S}^2\times \mathbb{S}^2$, $\mathbb{H}^2\times \mathbb{H}^2$, $\mathbb{S}^2\times \mathbb{R}^2$, $\phantom{a}\quad\mathbb{S}^2\times \mathbb{H}^2$, $\mathbb{H}^2\times \mathbb{R}^2$  & $\mathrm{SO}_2\times \mathrm{SO}_2$   &  6   &  Yes \\
    $\widetilde{\mathrm{Sl}}_2(\mathbb{R}) \times \mathbb{R}$, $\mathrm{Nil}_3 \times \mathbb{R}$, $\mathrm{Sol}^4_0$  & $\mathrm{SO}_2$  &   5 &  No \\
    $\mathrm{F}^4$  & $(\mathrm{S}^1)_{1,2}$  &   5 &  Yes \\
    $\mathrm{Nil}_4$, $\mathrm{Sol}^4_{m,n}$, $\mathrm{Sol^4_1}$  & $\{1\}$  &  4 &  No\\
    \hline %-------------------------------
  \end{tabular}
  \caption{List of Thurston four-dimensional geometries, their isotropy group (cf.~\cite[\S1]{Wall1986} and~\cite{Maier1998}), the dimension of their isometry group, and whether or not they admit a K{\"a}hler structure compatible with the geometric structure (cf.~\cite[Theorem~1.1]{Wall1986}). The spaces $\widetilde{\mathrm{Sl}}_2(\mathbb{R}) \times \mathbb{R}$, $\mathrm{Nil}_3 \times \mathbb{R}$, $\mathrm{Sol}^4_0$, and $\mathrm{Sol}_1^4$ are not K\"{a}hler but do admit complex structures. Here, $(\mathrm{S}^1)_{m,n}$ denotes the image of the unit circle $\mathbb{S}^1 \subset \mathbb{C}$ in $\mathrm{U}_2$ by $z \mapsto (z^m, z^n)$.}
  \label{tb:Thurston-4-geometries}
\end{table}
It is worth mentioning that \textsc{pmc} surfaces have been also studied in pseudo-Riemannian manifolds. A classification was achieved for non-degenerate \textsc{pmc} surfaces in the four-dimensional Lorentzian space forms~\cite{CV2009}. It turns out that, as in the Riemannian case, all \textsc{pmc} surfaces lie in three-dimensional submanifolds. This classification was afterwards extended to any codimension and any signature of the metric (see~\cite{Chen2009} for the spacelike case and \cite{Chen2010, FH2010} for the timelike case). In the sequel we will restrict ourselves to the Riemannian case.

\section{Definitions and first properties}

Let $M$ be an $n$-dimensional orientable Riemannian manifold with metric $\langle{\,,\,}\rangle$ and Levi-Civita connection $\overline{\nabla}$, and let $\phi:\Sigma\to M$ be an isometric immersion of an orientable Riemannian surface $\Sigma$. The tangent space $T_p\Sigma$ will be identified with $\mathrm{d}\phi(T_p\Sigma)\subset T_{\phi(p)}M$ in the sequel, so the metric on $\Sigma$ will also be denoted by $\langle{\,},{\,}\rangle$ since the immersion is isometric. Therefore $T_{\phi(p)} M$ admits an orthogonal decomposition $T_{\phi(p)}M=T_p\Sigma \oplus T_p^\perp \Sigma$, where $T^\perp\Sigma$ is the so-called \emph{normal bundle} of the immersion.  This leads to considering the space $\mathfrak{X}(\Sigma)$ of (tangent) vector fields, i.e., smooth sections of $T\Sigma$, and the space $\mathfrak{X}^\bot(\Sigma)$ of normal vector fields, i.e., smooth sections of $T^\perp\Sigma$. We will denote by $u^\top\in T_p\Sigma$ and $u^\perp\in T_p^\perp\Sigma$ the components of a vector $u\in T_{\phi(p)}M$ with respect to this decomposition.

Given a normal vector field $\eta\in\mathfrak{X}^\perp(\Sigma)$, we can define the shape operator associated with $\eta$ as the self-adjoint endomorphism $A_\eta:\mathfrak{X}(\Sigma)\to\mathfrak{X}(\Sigma)$ given by $A_\eta(X)=-(\overline{\nabla}_X\eta)^\top$. Then the second fundamental form $\sigma:\mathfrak{X}(M)\times\mathfrak{X}(M)\to\mathfrak{X}^\perp(\Sigma)$ satisfies $\langle\sigma(X,Y),\eta\rangle=\langle A_\eta(X),Y\rangle$ for all $X,Y\in\mathfrak{X}(\Sigma)$ and $\eta\in\mathfrak{X}^\perp(\Sigma)$. The mean curvature vector $H$ of the immersion at $p\in \Sigma$ is defined as $H(p)=\frac{1}{2}(\sigma(e_1,e_1)+\sigma(e_2,e_2))$, where $\{e_1,e_2\}$ is an orthonormal basis of $T_p\Sigma$.

The normal bundle $T^\perp \Sigma$ can also be endowed with a connection $\nabla^\perp: \mathfrak{X}(\Sigma) \times \mathfrak{X}^\perp(\Sigma)\to \mathfrak{X}^\perp(\Sigma)$ defined as $\nabla^\perp_X \eta = (\overline{\nabla}_X \eta)^\perp$ for all $X\in\mathfrak{X}(\Sigma)$ and $\eta\in\mathfrak{X}^\perp(\Sigma)$.  This connection is called the \emph{normal connection}, and gives rise to a \emph{curvature tensor} $R^\perp:\mathfrak{X}(\Sigma)\times\mathfrak{X}(\Sigma)\times\mathfrak{X}^\perp(\Sigma)\to\mathfrak{X}^\perp(\Sigma)$, given by
\begin{equation}\label{eqn:normal-curvature-tensor}
R^\perp(X, Y)\eta = \nabla^\perp_{X}\nabla^\perp_{Y}\eta-\nabla^\perp_{Y}\nabla^\perp_{X}\eta - \nabla^\perp_{[X,Y]}\eta,\quad X,Y\in\mathfrak{X}(\Sigma),\ \eta\in\mathfrak{X}^\perp(\Sigma).
\end{equation}

\begin{definition}[Parallel mean curvature immersion]\label{def:pmc}
An isometric immersion $\phi:\Sigma \to M$ is said to have parallel mean curvature (\textsc{pmc} for short) if its mean curvature vector $H \in \mathfrak{X}^\perp(\Sigma)$ is parallel in the normal bundle, i.e., $\nabla^\perp H = 0$, but not identically zero.
\end{definition}

\begin{remark}
The minimal case $H=0$ has been excluded from Definition~\ref{def:pmc} due to several reasons that will become clear after Lemma~\ref{lm:first-properties-pmc}. Essentially, it is not possible to define a natural orthonormal frame in the normal bundle if $H=0$, which is crucial for some of the arguments below.
\end{remark}

Although several results in higher codimension will be mentioned hereinafter, let us assume now that the codimension is $2$, where \textsc{pmc} surfaces enjoy additional properties. In the first place, we will furnish the normal bundle with a natural orientation provided that both $M$ and $\Sigma$ are oriented, and define a notion of curvature in the normal bundle.

\begin{definition}\label{def:orientation-normal-bundle}
A basis $\{\eta, \nu\}$ in $T_p^\perp \Sigma$ is said to be \emph{positively oriented} if and only if $\{e_1, e_2, \eta, \nu\}$ is positively oriented in $T_{\phi(p)}M$ whenever $\{e_1, e_2\}$ is  a positively oriented basis of $T_p\Sigma$.
\end{definition}

The \emph{normal curvature} of $\phi$ is the smooth function $K^\perp\in C^\infty(\Sigma)$ defined by
\begin{equation}\label{eqn:normal-curvature}
K^\bot(p)=\langle R^\bot(e_1,e_2)e_3,e_4\rangle,
\end{equation}
where $\{e_1, e_2, e_3, e_4\}$ is an orthonormal basis of $T_pM$ such that $\{e_1, e_2\}$ and $\{e_3, e_4\}$ are positively oriented bases in $T_p\Sigma$ and $T_p^\perp \Sigma$, respectively.

\begin{lemma}\label{lm:first-properties-pmc}
Let $\phi: \Sigma \rightarrow M$ be a \textsc{pmc} immersion. Then the mean curvature vector $H$ has constant length (in particular, $H$ never vanishes). If additionally the codimension is $2$, then:
  \begin{enumerate}
    \item[(i)] \label{lm:first-properties-pmc:item:def-tilde-H} There exists a unique parallel normal field $\widetilde{H}\in\mathfrak{X}^\perp(\Sigma)$ such that the global frame $\{\widetilde{H}/|H|, H/|H|\}$ is positively oriented and orthonormal in $T^\perp\Sigma$.

    \item[(ii)] \label{lm:first-properties-pmc:item:def-normal-curvature} $K^\perp$ is identically zero, i.e., the normal bundle is flat.

    \item[(iii)] The Ricci equation for the Riemann curvature tensor $\overline R$ of $M$ reads
    \begin{equation}\label{eq:Ricci}
      \langle\overline{R}(X, Y) H, \eta \rangle = \langle{[A_{\eta}, A_H]X},{Y}\rangle, \qquad X, Y \in \mathfrak{X}(\Sigma),\, \eta \in \mathfrak{X}^\perp(\Sigma).
    \end{equation}
  \end{enumerate}
\end{lemma}

\begin{remark}
Flatness of the normal bundle of a \textsc{pmc} surface is a typical property in codimension $2$. If the codimension is bigger than $2$, it is possible to define likewise the normal sectional curvature of the normal bundle, but it does not necessarily vanish for \textsc{pmc} surfaces.
\end{remark}

\begin{proof}
Since $H$ is parallel in the normal bundle we have
\[
X(|H|^2) = 2\langle\overline{\nabla}_X H,H\rangle = 2\langle\nabla^\perp_X H,H\rangle = 0,
\]
for all $X \in \mathfrak{X}(\Sigma)$, so $|H|$ is constant on $\Sigma$. As for (i), the normal bundle is orientable in the sense of Definition~\ref{def:orientation-normal-bundle}, so we can define a rotation $R_p$ of angle $\pi/2$ in $T_p^\perp \Sigma$ such that $\{\eta, R_p\eta\}$ is positively oriented for all $\eta \in T_p^\perp \Sigma$. This rotation leaves the normal bundle of $\Sigma$ invariant, and should not be confused with a possible complex structure on $M$.

Hence $\widetilde{H} = -RH$ is such that $\{\widetilde{H}/|H|, H/|H|\}$ is a positively oriented global orthonormal frame of the normal bundle. Moreover, $\widetilde H$ is also parallel since it has constant length and is orthogonal to the parallel vector field $H$. Given a positively oriented orthonormal frame $\{e_1,e_2\}$ in $T\Sigma$, we can consider $e_3 = \widetilde{H}/|H|$ and $e_4 = H/|H|$, so Equation~\eqref{eqn:normal-curvature-tensor} and the fact that $\widetilde H$ is parallel yield
\begin{equation}\label{lm:first-properties-pmc:eqn1}
\begin{aligned}
|H| R^\perp(e_1, e_2)e_3&= \nabla^\perp_{e_1}\nabla^\perp_{e_2}\widetilde{H} - \nabla^\perp_{e_2}\nabla^\perp_{e_1}\widetilde{H} - \nabla^\perp_{[e_1, e_2]}\widetilde{H} = 0,\\
|H| R^\perp(e_1, e_2)e_4&= \nabla^\perp_{e_1}\nabla^\perp_{e_2}{H} - \nabla^\perp_{e_2}\nabla^\perp_{e_1}{H} - \nabla^\perp_{[e_1, e_2]}{H} = 0.\\
\end{aligned}
\end{equation}
From~\eqref{eqn:normal-curvature} and the first equation in~\eqref{lm:first-properties-pmc:eqn1}, we get that $K^\perp\equiv 0$, so (ii) is proved. Finally, given $X, Y \in \mathfrak{X}(\Sigma)$ and $\xi, \eta \in \mathfrak{X}^\perp(\Sigma)$, the Ricci equation reads $\overline{R}(X,Y,\xi, \eta) = R^\perp(X,Y,\xi, \eta) - \langle{[A_\xi, A_\eta]X},{Y}\rangle$, so (iii) is a consequence of taking $\xi=H$ in the Ricci equation and of the second identity in~\eqref{lm:first-properties-pmc:eqn1}.
\end{proof}

\section{The relation between \textsc{cmc} and \textsc{pmc} surfaces}

Parallel mean curvature surfaces are often considered the natural generalization to higher codimension of \textsc{cmc} surfaces in three-manifolds, so a leading idea in the study of \textsc{pmc} surfaces in four-manifolds is to reduce the codimension and rely on results for \textsc{cmc} surfaces. Our first approach to this idea will consist in finding natural assumptions on a hypersurface $N$ of a four-manifold $M$ guaranteeing that any \textsc{cmc} surface immersed in $N$ has parallel mean curvature vector in $M$. This is evident if $N$ is totally geodesic, but this condition can be relaxed as the following result ensures.

\begin{proposition}\label{prop:cmc-totally-umbilical}
Let $N$ be a totally umbilical \textsc{cmc} hypersurface of a four-manifold $M$. Then every \textsc{cmc} surface immersed in $N$ is either \textsc{pmc} or minimal in $M$.
\end{proposition}

\begin{proof}
  Let $\phi: \Sigma \rightarrow N$ be a \textsc{cmc} immersion with second fundamental form $\widetilde{\sigma}$ and mean curvature vector $\widetilde{H}$. The immersion $\phi$ can also be regarded as an immersion into $M$, so let us denote by $\sigma$ and $H$ the second fundamental form and the mean curvature vector of the immersion $\phi:\Sigma\to M$, respectively. We will also define $\widehat{\sigma}$ and $\widehat{H}$ as the second fundamental form and the mean curvature vector of $N$ as a hypersurface of $M$, respectively.

  Since $\sigma = \widetilde{\sigma} + \widehat{\sigma}$, taking the trace on $\Sigma$ we get that $2H = 2\widetilde{H} + 3\widehat{H} - \widehat{\sigma}(\eta, \eta)$, where $\eta$ is a unit normal vector field to $\phi(\Sigma)$ tangent to $N$. Taking into account that $N$ is totally umbilical, i.e.,\ $\widehat{\sigma}(x,y) = \langle{x,y}\rangle\widehat{H}$ for all $x, y \in TN$, we finally get that $H = \widetilde{H} + \widehat{H}$. Taking the derivative of this last equation with respect to a tangent vector field $V \in \mathfrak{X}(\Sigma)$ gives
  \[
  \begin{split}
    \nabla^\perp_V H &= \bigl(\overline{\nabla}_V (\widetilde{H} + \widehat{H})\bigr)^{\perp} = (\overline{\nabla}_V \widetilde{H})^{\perp} + (\overline{\nabla}_V \widehat{H})^{\perp} \\
    &= \bigl( \nabla^N_V \widetilde{H} + \widehat{\sigma}(V, \widetilde{H}) \bigr)^{\perp} + (\overline{\nabla}_V \widehat{H})^{\perp} \\
    &= (\nabla_V^N \widetilde{H})^\perp + (\langle V,\widetilde{H}\rangle\widetilde{H})^\perp + (\overline{\nabla}_V \widehat{H})^{\perp} = (\overline{\nabla}_V \widehat{H})^{\perp},
  \end{split}
  \]
  where $\nabla^N$ is the Levi-Civita connection of $N$ and we have taken into account that $(\nabla_V^N \widetilde{H})^\perp=0$ since $\widetilde{H}$ has constant length. We distinguish two cases:
\begin{itemize}
 \item If $\widehat H=0$ ($N$ is a totally geodesic hypersurface of $M$), then $\nabla^\perp_V H=(\overline{\nabla}_V \widehat{H})^{\perp}=0$ and $H$ is parallel, so we are done.
 \item Assume now that $\widehat H\neq 0$. Observe that $\langle\overline{\nabla}_V \widehat{H},\widehat{H}\rangle=0$ since $\widehat H$ has constant length, so $(\overline{\nabla}_V \widehat{H})^{\perp}$ is proportional to a unit vector field $\eta$, normal to $\Sigma$, but tangent to $N$. Hence
  \[
  \begin{split}
    (\overline{\nabla}_V \widehat{H})^{\perp} &= \langle\overline{\nabla}_V \widehat{H},\eta\rangle\eta = -\langle \widehat{H}, \overline{\nabla}_V \eta \rangle \eta = -\langle \widehat{H},\widehat{\sigma}(V, \eta) \rangle\eta \\
    &= -\langle \widehat H , \langle V,\eta\rangle \widehat{H} \rangle\eta = 0,
    \end{split}
  \]
where we used again that $N$ is a totally umbilical hypersurface in $M$.\qedhere
\end{itemize}
\end{proof}

\begin{remark}~ \label{rmk:cmc-totally-umbilical-proof}
Under the assumptions of Proposition~\ref{prop:cmc-totally-umbilical}, the mean curvature vector $H$ of $\Sigma$ in $M$ is just the sum of the mean curvature vector $\widetilde{H}$ of $\Sigma$ in $N$ and the mean curvature vector $\widehat{H}$ of $N$ in $M$, i.e., we have the orthogonal decomposition $H = \widetilde{H} + \widehat{H}$. Hence $\Sigma$ is \textsc{pmc} if and only if $\nabla^{\perp}H = 0$ and $\Sigma$ is not minimal in $N$ or $N$ is not totally geodesic in $M$.
\end{remark}

Let us analyse how Proposition~\ref{prop:cmc-totally-umbilical} can be applied in different four-manifolds where totally umbilical surfaces are classified in order to construct \textsc{pmc} surfaces.
\begin{enumerate}
 \item In the space forms $\mathbb{R}^4$, $\mathbb{S}^4$ and $\mathbb{H}^4$, totally umbilical hypersurfaces have constant sectional curvature and constant mean curvature. Hence, the \textsc{pmc} surfaces provided by Proposition~\ref{prop:cmc-totally-umbilical} are \textsc{cmc} surfaces in the three-dimensional space forms $\mathbb{R}^3$, $\mathbb{S}^3$ or $\mathbb{H}^3$ embedded totally umbilically in the four-dimensional space form.

 \item There are no totally umbilical hypersurfaces in the complex space forms $\mathbb{C}\mathrm{P}^2$ and $\mathbb{C}\mathrm{H}^2$ \cite{TT63}. This is one of the difficulties when trying to produce examples of \textsc{pmc} immersions. In fact, there are no \textsc{pmc} spheres in $\mathbb{C}\mathrm{P}^2$ or in $\mathbb{C}\mathrm{H}^2$ (cf.\ Theorem~\ref{thm:pmc-spheres-complex-space-forms}).

 \item In $\mathbb{S}^3\times \mathbb{R}$ and $\mathbb{H}^3\times \mathbb{R}$ there are plenty of totally umbilical hypersurfaces since both spaces are locally conformally flat, but only the totally geodesic ones have constant mean curvature~\cite{MT2014}. Since totally geodesic submanifolds in a product are the product or totally geodesic submanifolds, we conclude that such totally geodesic hypersurfaces are locally congruent to $\mathbb{S}^3$, $\mathbb{H}^3$, $\mathbb{S}^2\times\mathbb{R}$, or $\mathbb{H}^2\times\mathbb{R}$.

 \item In a Riemannian product $\mathbb M^2(c_1) \times\mathbb M^2(c_2)$ of two surfaces of constant Gaussian curvatures $c_1$ and $c_2$, with $(c_1,c_2)\neq(0,0)$, the only totally umbilical hypersurfaces with constant mean curvature are totally geodesic. Hence, they are open subsets of products of one surface and a geodesic in the other surface. This was proven for $\mathbb{S}^2\times \mathbb{S}^2$ and $\mathbb{H}^2\times \mathbb{H}^2$, where both factors have the same curvature, in~\cite{TU12}, but the proof can easily be extended to the other cases.

 \item Consider $\mathbb{E}(\kappa,\tau)\times\mathbb{R}$, the Riemannian product of a homogeneous three-space with the Euclidean line. If $\kappa-4\tau^2 = 0$, the first factor has constant sectional curvature and the classification of totally umbilical hypersurfaces with constant mean curvature has been treated in item 3. If $\tau = 0$ (and $\kappa \neq 0$), the first factor is either $\mathbb S^2 \times \mathbb R$ or $\mathbb H^2 \times \mathbb R$, so the space under consideration is either $\mathbb S^2 \times \mathbb R^2$ or $\mathbb H^2 \times \mathbb R^2$, which have been treated in item 4. In all other cases, if was proven in \cite{ST09,VdV08} that there are no totally umbilical surfaces in $\mathbb{E}(\kappa,\tau)$, so the only totally umbilical hypersurfaces of $\mathbb{E}(\kappa,\tau)\times\mathbb{R}$ are open parts of the slices $\mathbb{E}(\kappa,\tau)\times\{t_0\}$. A more general result in $G\times\mathbb{R}$, where $G$ is a simply connected three-dimensional Lie group endowed with a left-invariant metric, follows from the classification of totally umbilical surfaces in $G$ (see~\cite{MS2015}).
\end{enumerate}

\section{Quadratic differentials and the classification of \textsc{pmc} spheres}
\label{sec:holomorphic-differentials}

As in the theory of \textsc{cmc} surfaces in homogeneous Riemannian three-manifolds, the existence of quadratic differentials that are holomorphic for \textsc{pmc} immersions comes in handy in some ambient four-manifolds. For instance, in symmetric four-manifolds such as
\begin{itemize}
  \item the space forms $\mathbb{R}^4$, $\mathbb{S}^4$ and $\mathbb{H}^4$,
  \item the complex hyperbolic and projective spaces $\mathbb{C}\mathrm{P}^2$ and $\mathbb{C}\mathrm{H}^2$,
  \item the Riemannian products $\mathbb{S}^2\times \mathbb{S}^2$ and $\mathbb{H}^2\times \mathbb{H}^2$,
\end{itemize}
it is possible to define two holomorphic quadratic differentials for \textsc{pmc} surfaces. It is also hitherto possible to define one holomophic quadratic differential in a few other cases, such as in $\mathbb M^3(c) \times \mathbb{R}$ and $\mathbb M^2(c_1) \times \mathbb M^2(c_2)$ (de Lira and Vit\'{o}rio~\cite{LV08} and Kowalczyk \cite{Kowalczyk2011}), or in Sasakian space forms (Rosenberg and Fetcu~\cite{FH2014}). This is instrumental, for instance, in the classification of \textsc{pmc} spheres in the aforementioned spaces, for the fact that a non-trivial holomorphic differential vanishes often gives precious information.

Throughout this section, we will consider a \textsc{pmc} immersion $\phi:\Sigma\rightarrow M$ of an oriented surface $\Sigma$ into a four-manifold $M$ with second fundamental form $\sigma$. As in the previous section, $\nabla$ and $\overline\nabla$ will denote the Levi-Civita connections in $\Sigma$ and $M$, respectively, and $\overline R$ will stand for the Riemann curvature tensor of $M$. Also, $z = x + iy$ will be a conformal parameter on $\Sigma$ with conformal factor $e^{2u}$, giving rise to the usual basic vectors $\partial_z = \frac{1}{2}(\partial_x - i\partial_y)$ and $\partial_{\bar{z}} = \frac{1}{2}(\partial_x + i\partial_y)$.

\begin{lemma}\label{lm:properties-conformal-parameter-and-second-fundamental-form}
Under the previous assumptions, the following formulae hold:
\begin{enumerate}
  \item[(i)] $\langle{\partial_z},{\partial_{\bar{z}}}\rangle = \frac{1}{2}e^{2u}$ and $\langle{\partial_z},{\partial_z}\rangle = 0$.
  \item[(ii)] $\nabla_{\partial_z} \partial_{\bar{z}} = \nabla_{\partial_{\bar{z}}}\partial_z = 0$ and $\nabla_{\partial_z} \partial_z = 2u_z \partial_z$.
  \item[(iii)] $2\sigma(\partial_{\bar{z}}, \partial_z) = e^{2u}H$.
  \item[(iv)] $\langle{\sigma(\partial_z, \partial_z)},{\eta}\rangle_{\bar{z}} = \langle{\overline{R}(\partial_{\bar{z}}, \partial_z)\partial_z},{\eta}\rangle$ for any parallel normal section $\eta$.
\end{enumerate}
\end{lemma}

\begin{proof}
(i) is a consequence of $z$ being a conformal parameter, (ii) is a direct computation using Koszul's formula and (iii) is straightforward from the definition of $\partial_z$ and $\partial_{\bar{z}}$. We prove (iv):
\[
\begin{split}
\langle{\sigma(\partial_z, \partial_z)},{\eta}\rangle_{\bar{z}} &= \langle{\nabla^\perp_{\partial_{\bar{z}}} \sigma(\partial_z, \partial_z)},{\eta}\rangle + \langle{\sigma(\partial_z, \partial_z)},{\nabla^\perp_{\partial_{\bar{z}}} \eta}\rangle \\
&= \langle{(\overline{\nabla}_{\partial_{\bar{z}}} \sigma)(\partial_z, \partial_z) + 2\sigma(\nabla_{\partial_{\bar{z}}}\partial_z, \partial_z)},{\eta}\rangle \\
&= \langle{(\overline{\nabla}_{\partial_z} \sigma)(\partial_{\bar{z}}, \partial_z) + \overline{R}(\partial_{\bar{z}}, \partial_z)\partial_z},{\eta}\rangle \\
&= \langle{\nabla^\perp_{\partial_z} \sigma(\partial_{\bar{z}}, \partial_z) - \sigma(\nabla_{\partial_z}\partial_{\bar{z}}, \partial_z) - \sigma(\partial_{\bar{z}}, \nabla_{\partial_z}\partial_z) + \overline{R}(\partial_{\bar{z}}, \partial_z)\partial_z},{\eta}\rangle \\
&= \langle{\nabla^\perp_{\partial_z} (\tfrac{1}{2} e^{2u} H) - u_z e^{2u}H + \overline{R}(\partial_{\bar{z}}, \partial_z)\partial_z},{\eta}\rangle = \langle{\overline{R}(\partial_{\bar{z}}, \partial_z)\partial_z},{\eta}\rangle,
\end{split}
\]
where we have taken into account (ii), (iii), the fact that $\eta$ is parallel, the definition of the covariant derivative of $\sigma$ and the Codazzi equation $(\overline{\nabla}_X \sigma)(Y, Z) - (\overline{\nabla}_Y \sigma)(X, Z) = (\overline{R}(X,Y)Z)^\perp$.
%$(\overline{\nabla}_X \sigma)(Y, Z) = \nabla^\perp_X \sigma(Y, Z) - \sigma(\nabla_X Y, Z) - \sigma(Y, \nabla_X Z)$, and the Codazzi equation $(\overline{\nabla}_X \sigma)(Y, Z) - (\overline{\nabla}_Y \sigma)(X, Z) = (\overline{R}(X,Y)Z)^\perp$.
\end{proof}

\subsection{Space forms}\label{subsec:holomorphic-differentials-space-forms}
Let $M = \mathbb M^4(c)$ be the space form of constant sectional curvature $c\in\mathbb{R}$, and define in conformal coordinates the quadratic differentials
\begin{equation}\label{eq:quadratic-differentials-space-forms}
\begin{aligned}
\Theta(z) &= \langle{\sigma(\partial_z, \partial_z)},{H}\rangle \mathrm{d} z\otimes \mathrm{d} z,\\
\widetilde{\Theta}(z) &= \langle\sigma(\partial_z, \partial_z),\widetilde{H}\rangle \mathrm{d} z\otimes \mathrm{d} z,\end{aligned}
\end{equation}
where $\widetilde{H}$ is given in Lemma~\ref{lm:first-properties-pmc}. Equation~\eqref{eq:quadratic-differentials-space-forms} defines globally $\Theta$ and $\widetilde{\Theta}$, i.e., their expressions do not depend upon the choice of the conformal parameter.

\begin{proposition}\label{prop:holomorphic-differentials-pmc-space-forms}
Let $\phi: \Sigma \rightarrow\mathbb  M^4(c)$ be a parallel mean curvature immersion of an oriented surface $\Sigma$. Then $\Theta$ and $\widetilde{\Theta}$ defined by~\eqref{eq:quadratic-differentials-space-forms} are holomorphic quadratic differentials.
\end{proposition}

\begin{proof}
Taking into account that $\langle{\overline{R}(\partial_{\bar{z}}, \partial_z)\partial_z},{\eta}\rangle$ is zero for any normal vector field $\eta$ in a space form, the statement follows from Lemma~\ref{lm:properties-conformal-parameter-and-second-fundamental-form}.
\end{proof}

\begin{theorem}[{Ferus~\cite{Ferus71}, see also~\cite[Theorem 2.2]{Hoffman1973}}]
\label{thm:pmc-spheres-space-forms}
Let $\phi: S \rightarrow \mathbb M^4(c)$ be a \textsc{pmc} immersion of a sphere $S$ in a space form. Then $\phi(S)$ is contained in a totally umbilical hypersurface of $\mathbb{M}^4(c)$ as a minimal surface.
\end{theorem}

\begin{proof}
  For illustration purposes, we will prove the case $c = 0$, that is, $\mathbb M^4(0) = \mathbb{R}^4$. Since $S$ is a sphere and $\phi$ is a \textsc{pmc} immersion, both $\Theta$ and $\widetilde{\Theta}$ defined in~\eqref{eq:quadratic-differentials-space-forms} vanish. Since $\Theta = 0$, we obtain that $A_H = |H|^2\mathrm{Id}$ (i.e., $\phi$ is \emph{pseudo-umbilical}).

  Arguing as in the classical proof that complete and connected totally umbilical surfaces in $\mathbb{R}^3$ are spheres or planes, we consider the function $f: S \to \mathbb{R}^4$ given by $f(p) = H_p + |H_p|^2\phi(p)$. For any tangent vector field $V \in \mathfrak{X}(S)$ we get
  \[
    V(f) = V(H + |H|^2\phi) = \overline{\nabla}_V H + |H|^2 V = -A_H V + \nabla^\perp_V H + |H|^2V = 0,
  \]
  by using the pseudo-umbilicity, and identifying $T_pS$ with its image by $\mathrm{d} \phi$ in $T_{\phi(p)}\mathbb M^4(c)$. Hence $f$ is constant $a\in \mathbb{R}^4$, so the immersion satisfies
  \[
    \left\lvert\phi - \frac{a}{|H|^2}\right\rvert^2 = \frac{1}{|H|^2}.
  \]
  This means that $\phi(S)$ is contained in a sphere $\mathbb{S}^3\subset\mathbb{R}^4$ of radius $1/|H|$, which is totally umbilical in $\mathbb{R}^4$ with mean curvature $|\widehat{H}| = |H|$. Thus the mean curvature $\widetilde{H} = H - \widehat{H}$ of $S$ as a surface of $\mathbb{S}^3$ is zero (observe that $H$ and $\widehat{H}$ have the same length, and $\widetilde{H}$ and $\widehat{H}$ are orthogonal, see Remark~\ref{rmk:cmc-totally-umbilical-proof}).
\end{proof}

\begin{remark}\label{rmk:pmc-space-forms-Theta-zero}
In the proof of Theorem~\ref{thm:pmc-spheres-space-forms} we have only used one of the holomorphic differentials associated to the \textsc{pmc} immersion to get the result. Nevertheless, both holomorphic differentials will be needed to get a complete classification of \textsc{pmc} immersions in space forms (cf.\ Theorem~\ref{thm:classification-pmc-space-forms}) as well as to classify \textsc{pmc} spheres in $\mathbb{S}^2\times \mathbb{S}^2$ and $\mathbb{H}^2\times \mathbb{H}^2$ (cf.\ Theorem~\ref{thm:pmc-spheres-Riemannian-products}).

Besides, \cite{AdCT2010} showed that the spheres are not the only \textsc{pmc} surfaces in space forms for which the quadratic differential $\Theta$ vanishes identically: there is also a complete non-flat example in $\mathbb{H}^n$ with non-negative Gaussian curvature (cf.\ Remark~\ref{rmk:pmc-S2xS2-H2xH2-Theta-zero}).
\end{remark}

\subsection{Complex hyperbolic and projective spaces}\label{subsec:holomorphic-differentials-complex-spaces}

Let us consider $M=\mathbb{C}\mathrm{M}^2(c)$, i.e., the complex projective or hyperbolic space of constant holomorphic curvature $c$, also including $\mathbb{C}^2 = \mathbb{C}\mathrm{M}^2(0)$. The situation in the complex space forms is quite similar to that of real space forms, due to the fact that Fetcu~\cite{Fetcu2010} defined a couple of holomorphic quadratic differentials associated with \textsc{pmc} immersions in $\mathbb{C}\mathrm{M}^2(c)$.

The Riemann tensor of these spaces reads
\begin{equation}\label{eq:Riemann-tensor-complex-space-forms}
  \overline{R}(X,Y)Z = \frac{c}{4}\Bigl\{ \langle{Y},{Z}\rangle X - \langle{X},{Z}\rangle Y + \langle{JY},{Z}\rangle JX - \langle{JX},{Z}\rangle JY - 2\langle{X},{JY}\rangle JZ\Bigr\},
\end{equation}
where $J:\mathfrak{X}(M) \rightarrow \mathfrak{X}(M)$ is the complex structure, which satisfies:
\begin{enumerate}
  \item $J^2 = -\mathrm{Id}$.
  \item $J$ is an isometry, i.e., $\langle{JX},{JY}\rangle = \langle{X},{Y}\rangle$.
  \item $J$ is parallel, i.e., $\overline{\nabla}_X JY = J\overline{\nabla}_X Y$, being $\overline{\nabla}$ the Levi-Civita connection of $\mathbb{C}\mathrm{M}^2(c)$.
\end{enumerate}

\begin{proposition} [{\cite[Proposition~2.3 and Section~3.1]{Fetcu2010}}]
Let $\phi: \Sigma \to \mathbb{C}\mathrm{M}^2(c)$ be a \textsc{pmc} immersion of an oriented suface $\Sigma$, and let $z = x+iy$ be a conformal parameter on $\Sigma$. Then
\begin{equation}\label{eq:Hopf-diferentials-complex-space-forms}
\begin{split}
\Theta(z) &= \bigl(8|H|^2 \langle\sigma(\partial_z, \partial_z),H\rangle) + 3c\langle J\phi_z,H\rangle^2\bigl)\mathrm{d} z \otimes \mathrm{d} z, \\
\widetilde{\Theta}(z) &= \bigl(8i|H|^2 \langle{\sigma(\partial_z, \partial_z),\widetilde{H}}\rangle) + 3c\langle J\phi_z,\widetilde{H}\rangle^2\bigl)\mathrm{d} z \otimes \mathrm{d} z,
\end{split}
\end{equation}
define two quadratic holomorphic differentials on $\Sigma$.
\end{proposition}
On the one hand, if $c=0$, then these differentials reduce to the corresponding diffentials in $\mathbb{C}^n\equiv\mathbb{R}^{2n}$. On the other hand, the appearance of the new extra term $\langle{J\Phi_z},{H}\rangle$ can be motivated by the fact that the Codazzi equation in $\mathbb{C}\mathrm{M}^2(c)$ is not as simple as in the case of $\mathbb M^4(c)$.

\begin{proof}
The holomorphicity follows easily from Lemma~\ref{lm:properties-conformal-parameter-and-second-fundamental-form}, from the expression of the Riemann tensor~\eqref{eq:Riemann-tensor-complex-space-forms} and from the following equalities:
\[
\langle{JH},{\widetilde{H}}\rangle = 2i|H|^2 e^{-2u}\langle{J\Phi_z},{\Phi_{\bar{z}}}\rangle, \quad \quad
(J\Phi_z)^\top = 2e^{-2u}\Phi_z.
\]
Let us justify the first one, by showing that if $\{e_1, e_2, e_3, e_4\}$ is an oriented orthonormal basis, then $\langle{Je_1},{e_2}\rangle = \langle{Je_3},{e_4}\rangle$. Let $C = \langle{Je_1},{e_2}\rangle$, which satisfies $C^2\leq 1$ by Cauchy-Schwarz inequality. If $C^2=1$, then $Je_1=\pm e_2$, so $Je_3=\pm e_4$ and we are done. If $C^2<1$, let us define $\widetilde{e}_3 = (1-C^2)^{-1/2}(Ce_1 + Je_2)$ and $\widetilde{e}_4 = (1-C^2)^{-1/2}(Je_1 - Ce_2)$. Then $\{\widetilde{e}_3, \widetilde{e}_4\}$ is an oriented orthonormal basis spanning the same plane as $\{e_3, e_4\}$, so they differ in a rotation of angle $\theta$, and it is easy to check that $\langle{Je_3},{e_4}\rangle = \langle{J\widetilde{e}_3},{\widetilde{e}_4}\rangle = C$.
\end{proof}

Although there exist two holomorphic quadratic differentials, there is no direct proof of the classification of the \textsc{pmc} spheres in $\mathbb{C}\mathrm{M}^2(c)$. All the known proofs use the structure equations for \textsc{pmc} surfaces in $\mathbb{C}\mathrm{M}^2(c)$ provided by Ogata~\cite{Ogata95}. The proof given by Fetcu in~\cite[Corollary~3.2]{Fetcu2010} uses the two holomorphic differentials to show that such a sphere must have constant Gauss curvature, so the result follows from~\cite[Theorem~1.1]{Hirakawa2006}.

\begin{theorem}[{\cite[Corollary 1.2]{Hirakawa2006} and also \cite[Corollary~3.2]{Fetcu2010}}]
\label{thm:pmc-spheres-complex-space-forms}
Let $\phi: S \to \mathbb{C}\mathrm{M}^2(c)$ a \textsc{pmc} immersion of a sphere $S$. Then $c = 0$ and $S$ is a round sphere in a hyperplane of $\mathbb{C}^2$.
\end{theorem}

This non-existence result of \textsc{pmc} spheres in $\mathbb{C}\mathrm{H}^2$ and $\mathbb{C}\mathrm{P}^2$ contrasts with the rest of the symmetric spaces, where there do exist \textsc{pmc} spheres (cf.\ Theorem~\ref{thm:pmc-spheres-space-forms} and Theorem~\ref{thm:pmc-spheres-Riemannian-products}). In other Thurston four-geometries like $\mathbb{M}^3(c)\times\mathbb{R}$, $\mathbb M^2(c_1)\times\mathbb M^2(c_2)$, $\mathbb{E}(\kappa,\tau)\times\mathbb{R}$ or $\mathrm{Sol}_3\times\mathbb{R}$, there always exist \textsc{pmc} spheres, since $\mathbb{H}^3$ and the $\mathbb{E}(\kappa,\tau)$-spaces or $\mathrm{Sol}_3$ do admit \textsc{cmc} spheres (see the comments below Proposition~\ref{prop:cmc-totally-umbilical}).

\subsection[Riemannian products]{The Riemannian products $\mathbb{S}^2\times \mathbb{S}^2$ and $\mathbb{H}^2\times \mathbb{H}^2$}
\label{subsec:holomorphic-differentials-Riemannian-products}

Now let $M=\mathbb M^2(\epsilon) \times\mathbb M^2(\epsilon)$, where $\mathbb M^2(\epsilon)$ stands for the $2$-sphere $\mathbb{S}^2$ ($\epsilon = 1$) or the hyperbolic plane $\mathbb{H}^2$ ($\epsilon = -1$). Since both $\mathbb{S}^2$ and $\mathbb{H}^2$ admit a complex structure $J$, we can define on $M$ two different (but equivalent) complex structures $J_1 = (J, J)$ and $J_2 = (J, -J)$ (see~\cite[Section~3]{TU12}). Moreover, we can define a \emph{product structure} $P: T M \rightarrow TM$ as $P(u,v) = (u,-v)$, which enjoys the following properties:
\begin{enumerate}
  \item $P$ is a self-adjoint linear involutive isometry of every tangent plane of $M$.
  \item $J_2 = PJ_1 = J_1P$
  \item $P$ is parallel, i.e., $\overline{\nabla}_X PY = P\overline{\nabla}_X Y$ for all $X,Y\in\mathfrak{X}(M)$.
\end{enumerate}
The operator $P$ allows us to write the Riemann tensor of $\mathbb M^2(\epsilon)\times \mathbb M^2(\epsilon)$ as
\begin{equation}\label{eq:Riemann-tensor-Riemannian-products}
\overline{R}(X,Y)Z = \frac{\epsilon}{2} \bigl[ \langle{Y},{Z}\rangle X - \langle{X},{Z}\rangle Y + \langle{Y},{PZ}\rangle PX - \langle{X},{PZ}\rangle PY \bigr].
\end{equation}
In particular, $\mathbb M^2(\epsilon)\times\mathbb M^2(\epsilon)$ is an Einstein manifold of constant scalar curvature $4 \epsilon$ (this is no longer true in the general case $\mathbb{M}^2(c_1)\times\mathbb{M}^2(c_2)$). The existence of two holomorphic differential was shown in~\cite{TU12}.

\begin{proposition}[{\cite[Proposition~3]{TU12}}]\label{prop:differentials-M2xM2}
Let $\phi: \Sigma\to\mathbb M^2(\epsilon)\times\mathbb M^2(\epsilon)$ be a \textsc{pmc} immersion of an oriented suface $\Sigma$, and let $z = x+iy$ be a conformal parameter on $\Sigma$. Then
\begin{equation}\label{eq:holomorphic-differentials-Riemannian-products}
  \begin{split}
    \Theta_1(z) &= \left(2\langle{\sigma(\partial_z, \partial_z)},{H + i \widetilde{H}}\rangle + \frac{\epsilon}{4|H|^2}\langle{J_1 \phi_z},{H + i\widetilde{H}}\rangle^2 \right) \mathrm{d} z \otimes \mathrm{d} z, \\
    \Theta_2(z) &=  \left(2\langle{\sigma(\partial_z, \partial_z)},{H - i \widetilde{H}}\rangle + \frac{\epsilon}{4|H|^2}\langle{J_2 \phi_z},{H - i\widetilde{H}}\rangle^2 \right) \mathrm{d} z \otimes \mathrm{d} z, \\
  \end{split}
\end{equation}
are two holomorphic quadratic differentials.
\end{proposition}

\begin{proof}
It also follows from Lemma~\ref{lm:properties-conformal-parameter-and-second-fundamental-form} after some manipulations, as in the previous cases.
\end{proof}

\begin{remark}
  The differentials $\Theta_1$ and $\Theta_2$ can be chosen in different ways, since any linear combination of them is also holomorphic. As a particular case and taking into account that $\langle{J_1 \phi_z},{H}\rangle = i\langle{J_1\phi_z},{\widetilde{H}}\rangle$, $\langle{J_2\phi_z},{H}\rangle = -i\langle{J_2\phi_z},{\widetilde{H}}\rangle$, and $J_2 = PJ_1$, we can define the following two holomorphic quadratic differentials (cf.\ equation~\eqref{eq:Hopf-diferentials-complex-space-forms}):
  \begin{equation}\label{eq:holomorphic-differentials-Riemannian-products-var}
    \begin{split}
      \Theta &= \Bigl(4|H|^2\langle{\sigma(\partial_z, \partial_z)},{H}\rangle + \epsilon\bigl[\langle{J_1\phi_z},{H}\rangle^2 + \langle{J_1\phi_z},{PH}\rangle^2\bigr]\Bigr) \mathrm{d} z\otimes \mathrm{d} z \\
     \widetilde{\Theta} &= \Bigl(4i|H|^2\langle{\sigma(\partial_z, \partial_z)},{\widetilde{H}}\rangle - \epsilon\bigl[\langle{J_1\phi_z},{\widetilde{H}}\rangle^2 - \langle{J_1\phi_z},{P\widetilde{H}}\rangle^2\bigr]\Bigr) \mathrm{d} z \otimes \mathrm{d} z \\
    \end{split}
  \end{equation}
  It is easy to show that $\Theta= |H|^2(\Theta_1 + \Theta_2)$ and $\widetilde\Theta=|H|^2(\Theta_1 - \Theta_2)$, so these expressions make it clear that $\Theta_1$ and $\Theta_2$ extend the classical differentials in $\mathbb{R}^4$ given by~\eqref{eq:quadratic-differentials-space-forms}.
\end{remark}

Using that these two differentials vanish on spheres, it is shown in~\cite{TU12} that the extrinsic normal curvature of an immersed \textsc{pmc} sphere has to be zero. Then the following classification is a consequence of Theorem~\ref{thm:clasification-pmc-M2xM2}.

\begin{theorem}[{\cite[Corollary 1]{TU12}}]
\label{thm:pmc-spheres-Riemannian-products}
Let $\phi: S \to\mathbb M^2(\epsilon) \times\mathbb M^2(\epsilon)$, $\epsilon^2 = 1$, be a \textsc{pmc} immersion of a sphere $S$. Then $\phi$ is a \textsc{cmc} sphere in a totally geodesic hypersurface of $\mathbb M^2(\epsilon) \times\mathbb M^2(\epsilon)$.
\end{theorem}

\begin{remark}\label{rmk:pmc-S2xS2-H2xH2-Theta-zero}
It is interesting to highlight that \textsc{pmc} spheres are not the only surfaces with vanishing holomorphic differentials. Indeed, the product of two hypercycles in $\mathbb{H}^2\times\mathbb{H}^2$ with curvatures satisfying $k_1^2 + k_2^2 = 1$ and a special embedding of the hyperbolic plane in $\mathbb{H}^2\times \mathbb{H}^2$ also satisfy that condition (see~\cite[Theorem~4]{TU12}).
\end{remark}

\subsection[The Riemann product of a space form and the real line]{The Riemannian products $\mathbb M^3(c) \times \mathbb{R}$} \label{subsec:holomorphic-differentials-M3xR}

The study of \textsc{pmc} surfaces in $M=\mathbb M^3(c)\times \mathbb{R}$ was tackled by de Lira and Vit\'{o}rio~\cite{LV08}, as well as by Alencar, Do Carmo and Tribuzy~\cite{AdCT2010}. As in all the previous cases these authors found a holomorphic quadratic differential. In spite of their claim that there are two holomorphic differentials $Q^h$ and $Q^v$, a deeper analysis shows that $Q^h$ and $Q^v$ coincide.

The Riemann tensor of $\mathbb M^3(c) \times \mathbb{R}$ is given by:
\begin{equation}\label{eq:Riemann-tensor-M3xR}
\begin{split}
\overline{R}(X,Y)Z &= \frac{c}{4} \bigl( \langle{Y+PY},{Z}\rangle(X+PX) - \langle{X+PX},{Z}\rangle(Y+PY)\bigr) \\
  &= c\Bigl(\langle{Y},{Z}\rangle X - \langle{X},{Z}\rangle Y - \langle{Y},{\zeta}\rangle\langle{Z},{\zeta}\rangle X + \langle{X},{\zeta}\rangle\langle{Z},{\zeta}\rangle Y + \\
  &\quad + \langle{X},{Z}\rangle\langle{Y},{\zeta}\rangle\zeta - \langle{Y},{Z}\rangle\langle{X},{\zeta}\rangle\zeta\Bigr),
\end{split}
\end{equation}
where $P$ is the product structure in $T (\mathbb M^3(c) \times \mathbb{R}) \equiv T\mathbb M^3(c) \times \mathbb{R}$ given by $P(u, t) = (u,-t)$ for all $u \in T\mathbb{M}^3(c)$, and $t \in \mathbb{R}$, and $\zeta$ is a unit tangent vector to the factor $\mathbb{R}$. The second expression in~\eqref{eq:Riemann-tensor-M3xR} follows from the identity $PX = X - 2\langle{X},{\zeta}\rangle\zeta$ for all $X\in\mathfrak{X}(M)$.

\begin{proposition}
Let $\phi:\Sigma \rightarrow\mathbb M^3(c) \times \mathbb{R}$ be a \textsc{pmc} immersion of an oriented surface $\Sigma$ and let $z = x+iy$ be a conformal parameter. Then
\begin{equation}\label{eq:holomorphic-differential-space-formxR}
\begin{split}
\Theta(z) &= \bigl( 2\langle{\sigma(\partial_z, \partial_z)},{H}\rangle - c\langle{\phi_z},{\zeta}\rangle^2\bigr) \mathrm{d} z \otimes \mathrm{d} z \\
   &= \bigl( 2\langle{\sigma(\partial_z, \partial_z)},{H}\rangle + \tfrac{c}{2}\langle{\phi_z},{P\phi_z}\rangle\bigr) \mathrm{d} z \otimes \mathrm{d} z
\end{split}
\end{equation}
is a holomorphic quadratic differential in $\Sigma$.
\end{proposition}

\begin{proof}
  Both expressions for $\Theta$ coincide, which follows from the equality $P\phi_z = \phi_z - 2\langle{\phi_z},{\zeta}\rangle\zeta$ and the fact that $z$ is a conformal parameter, i.e., $\langle{\phi_z},{\phi_z}\rangle = 0$. Using now Lemma~\ref{lm:properties-conformal-parameter-and-second-fundamental-form} and the second equality in~\eqref{eq:Riemann-tensor-M3xR}, we deduce that $\langle{\sigma(\partial_z,\partial_z)},{H}\rangle_{{\bar{z}}} = \tfrac{c}{2} e^{2u}\langle{\phi_z},{\zeta}\rangle\langle{H},{\zeta}\rangle$, and also
  \[
    (\langle{\phi_z},{\zeta}\rangle^2)_{\bar{z}} = 2\langle{\phi_z},{\zeta}\rangle\langle{\overline{\nabla}_{\partial_{\bar{z}}} \phi_z},{\zeta}\rangle =  e^{2u}\langle{\phi_z},{\zeta}\rangle\langle{H},{\zeta}\rangle,
  \]
  where we have taken into account that $\zeta$ is a parallel vector field. Consequently, the differential is holomorphic.
\end{proof}

De Lira and Vit\'{o}rio use this quadratic differential $\Theta$ to classify the \textsc{pmc} spheres in $\mathbb M^3(c)\times \mathbb{R}$ by showing that there is a principal frame $\{e_1,e_2\}$ on the surface such that the the associated curvature lines to $e_1$ lie in horizontal slices. Then an analysis of these curvature lines leads to the following result:

\begin{theorem}[{\cite[Theorem~3.2]{LV08}}]
\label{thm:pmc-spheres-space-formxR}
The only \textsc{pmc} spheres immersed in $\mathbb M^3(c) \times \mathbb{R}$ are the rotationally invariant \textsc{cmc} surfaces embedded in totally geodesic cylinders $\mathbb M^2(c) \times \mathbb{R}$ or in totally geodesic slices $\mathbb M^3(c)\times \{t_0\}$, $t_0 \in \mathbb{R}$.
\end{theorem}

A result of the same kind is obtained by Alencar, do Carmo and Tribuzy in $\mathbb{M}^4(c) \times \mathbb R$ (codimension 3), as we show next. One expects that a \textsc{pmc} sphere in $\mathbb{M}^4(c) \times \mathbb R$ lies either in a slice $\mathbb M^4(c) \times\{t_0\}$ or in some $\mathbb M^2(c) \times \mathbb{R}$ as a \textsc{cmc} sphere (hence rotationally invariant). Unfortunately, a further reduction of the codimension still remains an open problem, which would give the complete classification of \textsc{pmc} spheres in $\mathbb M^n(c) \times \mathbb{R}$ for all $n\geq 4$ (see Theorem~\ref{thm:classification-pmc-surfaces-space-formxR}).

\begin{theorem}[{\cite[Theorem 2]{AdCT2010}}]
\label{thm:pmc-spheres-M4xR}
Let $\phi: S \rightarrow\mathbb M^4(c) \times \mathbb{R}$ be a \textsc{pmc} immersion of a sphere $S$. Then one of the following assertions holds:
\begin{enumerate}
  \item[(i)] $\phi(S)$ is contained in a totally umbilical hypersurface of $\mathbb M^4(c) \times \{t_0\}$ as a \textsc{cmc} surface.

  \item[(ii)] Considering $\mathbb M^4(c)\times\mathbb{R}$ isometrically embedded in $\mathbb{R}^6$ ($c=1$) or $\mathbb{R}^6_1$ ($c=-1$), there is a plane $\Pi$ such that $\phi(S)$ is invariant under rotations which fix $\Pi^\bot$, and the level curves of the height function $p \mapsto \langle{\phi(p)},{\zeta}\rangle$ are circles lying in planes parallel to $\Pi$.
\end{enumerate}
\end{theorem}

\begin{remark}\label{rmk:improvement-spheres-MxR}
Mendon\c{c}a and Tojeiro~\cite{MT2014} improve item (ii) in the previous result by showing that, in general codimention, $\phi(\Sigma)$ is a rotationally surface in a totally geodesic $\mathbb M^m(c) \times \mathbb{R}$, $m \leq 4$, over a curve in a totally geodesic $\mathbb M^s(c)\times \mathbb{R}$, $s \leq 3$.
\end{remark}

\subsection{The Riemannian products $\mathbb M^2(c_1) \times \mathbb M^2(c_2)$.}
\label{subsec:holomorphic-differentials-S2xH2}

Let us finally consider $M =\mathbb M^2(c_1) \times \mathbb M^2(c_2)$. Following the notation introduced in Section~\ref{subsec:holomorphic-differentials-Riemannian-products}, the Riemann tensor of $M$ can be expressed as
\[
\overline{R}(X,Y)Z = c_1 R_0(P_1X,P_1Y)Z + c_2 R_0(P_2X,P_2Y)Z,
\]
where $R_0(X,Y)Z = \langle{Y},{Z}\rangle X - \langle{X},{Z}\rangle Y$, $P_1 = \frac 12 (I+P)$ and $P_2 = \frac 12 (I-P)$ are the projections to the factors, i.e., $P_1(u,v)=(u,0)$ and $P_2(u,v)=(0,v)$.

De Lira and Vit\'{o}rio~\cite{LV08} defined a holomorphic quadratic differential for \textsc{pmc} surfaces in $\mathbb{S}^2\times \mathbb{H}^2$ (where the constant Gauss curvatures of the factors are exactly opposite) and the holomorphicity of this differential also follows from the ideas in~\cite{TU12}. Kowalczyk \cite{Kowalczyk2011} extended this by defining a quadratic differential in the general case of $\mathbb M^2(c_1) \times\mathbb M^2(c_2)$, cf.\ the next proposition. In contrast to the previous cases, the classification of \textsc{pmc} spheres in $\mathbb M^2(c_1) \times\mathbb M^2(c_2)$ is still an open problem, even in $\mathbb{S}^2\times \mathbb{H}^2$. The natural candidates are those given by Proposition~\ref{prop:cmc-totally-umbilical}, i.e., \textsc{cmc} spheres immersed in totally geodesic hypersurfaces of $\mathbb M^2(c_1) \times\mathbb M^2(c_2)$.

\begin{proposition}\label{prop:holomorphic-differential-M2xM2}
Let $\phi: \Sigma \to\mathbb M^2(c_1) \times\mathbb M^2(c_2)$ be a \textsc{pmc} immersion of an oriented surface $\Sigma$ and $z = x+iy$ a conformal parameter. Then
\begin{equation}\label{eq:holomorphic-differential-M2xM2}
    \begin{split}
    \Theta(z) = \Bigl( 2|{H}|^2\langle{\sigma(\partial_z, \partial_z)},{H}\rangle &+ c_1 \langle R_0(P_1\phi_z,P_1H)H,\phi_z \rangle \\
    &- c_2 \langle R_0(P_2\phi_z,P_2H)H,\phi_z \rangle \Bigr) \mathrm{d} z\otimes \mathrm{d} z
    \end{split}
\end{equation}
is a holomorphic quadratic differential on $\Sigma$.
\end{proposition}

In the case $c_1=c_2=\pm 1$, the holomorphic differential given by~\eqref{eq:holomorphic-differential-M2xM2} is a linear combination of the two holomorphic differentials in Proposition~\ref{prop:differentials-M2xM2}.

\section{The general non-spherical case}

Proposition~\ref{prop:cmc-totally-umbilical} reveals that the ambient spaces considered above are plentiful of \textsc{pmc} immersions in general: any \textsc{cmc} immersion into a totally umbilical \textsc{cmc} hypersurface is \textsc{pmc}. Nonetheless, this description does not give all \textsc{pmc} surfaces in general, as examples in complex space forms or in product manifolds $\mathbb{M}^2(\epsilon)\times\mathbb{M}^2(\epsilon)$ below show. On account of the fact that listing all \textsc{pmc} surfaces is not reasonable, instead local classification results have been considered so far, based either on reducing the codimension to the \textsc{cmc} case (space form cases and $\mathbb M^n(c)\times \mathbb{R}$), or on associating some analytic data with the immersion (complex hyperbolic and projective spaces, see also Hoffman's examples~\cite[Theorem 5.1]{Hoffman1973} in $\mathbb{R}^4$ at the end of Section~\ref{subsec:general-M4}). We will present as well results with extra conditions on the immersion.

\subsection{PMC surfaces in space forms}\label{subsec:general-M4}
\label{subsec:pmc-surfaces-space-forms}
Chen~\cite{Chen73} classified \textsc{pmc} surfaces in Euclidean space $\mathbb{R}^4$, and Yau~\cite{Yau74} gave an independent classification in an arbitrary space form $\mathbb{M}^4(c)$.

\begin{theorem}[{\cite[Theorem 4]{Yau74}}]\label{thm:classification-pmc-space-forms}
Let $\phi: \Sigma \to\mathbb M^4(c)$ be a \textsc{pmc} immersion of an oriented surface $\Sigma$. Then $\Sigma$ is contained in a totally umbilical hypersurface of $\mathbb M^4(c)$ as a \textsc{cmc} surface.
\end{theorem}

\begin{remark}
  Although Theorem~\ref{thm:classification-pmc-space-forms} is stated in dimension four, Chen and Yau proved this result in arbitrary dimension, showing, more precisely that either $\phi$ is minimal in a totally umbilical hypersurface of $\mathbb{M}^n(c)$, or $\phi$ is a \textsc{cmc} immersion into a totally umbilical three-dimensional submanifold of $\mathbb M^n(c)$.
\end{remark}

\begin{proof}
The idea is to use both differentials defined by~\eqref{eq:quadratic-differentials-space-forms} to show the existence of a parallel normal section $\xi$ such that $A_\xi = \lambda\,\mathrm{Id}$, and the same argument as in the proof of Theorem~\ref{thm:pmc-spheres-space-forms} will ensure that $\Sigma$ satisfies the desired conditions
%($\lambda=0$ corresponds to the minimal case, whereas $\lambda\neq 0$ leads to the \textsc{cmc} case).
To illustrate this, let us assume $c=0$.

If $\Theta = 0$, then $A_H = |H|^2\,\mathrm{Id}$ and we can reason as in the proof of Theorem~\ref{thm:pmc-spheres-space-forms}. Likewise, if $\widetilde{\Theta} = 0$, then $A_{\widetilde{H}} = \lambda\, \mathrm{Id}$ with $\lambda = \langle{H},{\widetilde{H}}\rangle = 0$ so $p\mapsto \tilde{H}_p$ is constant in $\mathbb{R}^4$ since $V(\widetilde{H}) = -A_{\widetilde{H}}V + \nabla^\perp_V \widetilde{H} = 0$ for all $V\in \mathfrak{X}(\Sigma)$. The function $f: \Sigma \rightarrow \mathbb{R}$ defined as $f(p) = \langle\phi(p) - \phi(p_0), \widetilde{H}\rangle$ for some $p_0\in\Sigma$ satisfies
\[
V(f) = \langle{V},{\widetilde{H}}\rangle + \langle{\phi(p) - \phi(p_0)},{\overline{\nabla}_V \widetilde{H}}\rangle = 0,\quad \text{for all } V \in \mathfrak{X}(\Sigma),
%\langle{\phi(p)-\phi(p_0)},{-A_{\widetilde{H}} V + \nabla^\perp_V \widetilde{H}}\rangle = 0,
\]
so $f$ is constant and $\phi(\Sigma)$ lies in a hyperplane of $\mathbb{R}^4$. Moreover, $\phi(\Sigma)$ has constant mean curvature in this hyperplane

Hence we can assume that $\Theta$ and $\widetilde{\Theta}$ are not identically zero. It is not hard to prove that the imaginary part of the meromorphic function $g: \Sigma \to \mathbb{C}$, $g(p) = \Theta(p)/\widetilde{\Theta}(p)$, coincides with the commutator $[A_H,A_{\widetilde{H}}]$, which is zero by the Ricci equation~\eqref{eq:Ricci}. Hence the imaginary part of $g$ identically vanishes, whence $g\equiv\tan(\alpha)$ for some constant $|\alpha|<\frac{\pi}{2}$. The normal vector field $\xi=\cos(\alpha) H-\sin(\alpha) \widetilde{H}$ is parallel and satisfies $A_\xi = \cos\alpha|H|^2 \, \mathrm{Id}$, so we can again continue as in the proof of Theorem~\ref{thm:pmc-spheres-space-forms}, considering the function $f(p) = \xi_p + \cos \alpha |H|^2\phi(p)$.
\end{proof}

For illustrative purposes, let us consider Lawson's minimal examples~\cite[Theorem~2]{Lawson70} in $\mathbb{S}^3$ as a \textsc{pmc} surfaces in $\mathbb{R}^n$, $n \geq 4$. Hence, any compact orientable surface of genus $g$ can be embedded as a \textsc{pmc} surface in $\mathbb{R}^n$, $n \geq 4$. Hoffman gave more examples of \textsc{pmc} surfaces in the space forms not lying in any hypersphere as minimal surfaces~\cite[Theorem~5.1]{Hoffman1973}. More particularly, he showed that, given any holomorphic function $\varphi:U\to\mathbb{C}$ on an open domain $U\subseteq\mathbb{C}$, and constants $H>0$ and $\alpha \in \mathbb{R}$ there exists a \textsc{pmc} immersion in $\mathbb M^4(c)$ such that the length of the mean curvature vector is $H$, $\Theta = \varphi (\mathrm{d} z)^2$ and $\widetilde{\Theta} = \alpha\varphi(\mathrm{d} z)^2$.

\subsection{PMC surfaces in $\mathbb{S}^3\times \mathbb{R}$ and $\mathbb{H}^3\times \mathbb{R}$}\label{subsec:general-M3xR}
\label{subsec:pmc-surfaces-M3xR}
Alencar, Do Carmo and Tribuzy~\cite{AdCT2010} studied \textsc{pmc} immersions in $\mathbb{M}^n(c) \times \mathbb{R}$, $c \neq 0$, for arbitrary $n$. They realized that the quadratic differential~\eqref{eq:holomorphic-differential-space-formxR} introduced by de Lira and Vit\'{o}rio~\cite{LV08} is holomorphic for any $n \geq 2$ (for $n = 2$ it is actually the Abresch-Rosenberg differential~\cite{AR05}). They showed that for a \textsc{pmc} immersion in $\mathbb M^n(c)\times \mathbb{R}$ either $H$ is an umbilical direction, i.e., $A_H = |H|^2\mathrm{Id}$ (so $\phi(\Sigma)$ lies in a slice $\mathbb{M}^n(c) \times \mathbb{R}$, see items (i) and (ii) in Theorem~\ref{thm:classification-pmc-surfaces-space-formxR}), or one can reduce the codimension to three.

\begin{theorem}[{\cite[Theorem~1]{AdCT2010}}]
\label{thm:classification-pmc-surfaces-space-formxR}
Let $\phi:\Sigma\to\mathbb{M}^n(c)\times \mathbb{R}$ a \textsc{pmc} immersion of an oriented surface $\Sigma$. Then, one of the following assertions holds:
\begin{enumerate}
  \item[(i)] $\phi(\Sigma)$ is minimal in a totally umbilical hypersurface of $\mathbb{M}^n(c)\times\{t_0\}$, $t_0 \in \mathbb{R}$.

  \item[(ii)] $\phi(\Sigma)$ is \textsc{cmc} in a three-dimensional totally umbilical submanifold of $\mathbb{M}^n(c)\times\{t_0\}$, $t_0 \in \mathbb{R}$.

  \item[(iii)] If $n\geq 4$, then $\phi(\Sigma)$ lies in a totally geodesic $\mathbb{M}^4(c) \times \mathbb{R}$.
\end{enumerate}
\end{theorem}

\begin{remark}\label{rmk:classification-pmc-surfaces-space-formxR}
  Notice that Theorem~\ref{thm:classification-pmc-surfaces-space-formxR} does not provide a classification of \textsc{pmc} surfaces in $\mathbb{M}^3(c) \times \mathbb{R}$ or in $\mathbb{M}^4(c)\times\mathbb{R}$. Therefore the final classification of \textsc{pmc} surfaces in $\mathbb M^n(c)\times\mathbb{R}$ depends upon the cases  $n = 3$ and $n = 4$, which remain open. Moreover, \textsc{pmc} spheres have only been classified for $n = 3$ (cf.\ Theorem~\ref{thm:pmc-spheres-space-formxR}), though it is proven that they must be \emph{rotationally invariant} for $n = 4$~\cite[Theorem 2]{AdCT2010}.
\end{remark}

Mendon\c{c}a and Tojeiro have also discussed \textsc{pmc} immersions in $\mathbb{M}^n(c)\times \mathbb{R}$ in~\cite{MT2014}. They obtained more information adding an extra hypothesis. They show that, if $\phi(\Sigma)$ is not contained in a slice (cases (i) and (ii) in Theorem~\ref{thm:classification-pmc-surfaces-space-formxR}) and $\Theta \equiv 0$ then $\phi(\Sigma)$ is rotationally invariant in the sense exposed in Remark~\ref{rmk:improvement-spheres-MxR}. In particular, this condition is fulfilled if either $\Sigma$ is diffeomorphic to a sphere or $\Sigma$ is a complete non-flat surface in $\mathbb{H}^n\times \mathbb{R}$ with non-negative Gaussian curvature (cp.\ Remarks~\ref{rmk:pmc-space-forms-Theta-zero} and~\ref{rmk:pmc-S2xS2-H2xH2-Theta-zero}).

In order to prove the latter assertion, observe that if $\Theta \not\equiv 0$, then $\Delta \log |\Theta| = 4K \geq 0$, i.e., $\log|\Theta|$ is a superharmonic function bounded from below in $\Sigma$. Since $K\geq 0$ it follows that $\Sigma$ has quadratic area growth, so $\log|\Theta|$ must be constant in view of~\cite[Corollary~1]{CY75}. From the fact that $|\Theta|$ is constant, it follows that $K$ is also constantly zero.

This idea was previously developed by Hoffman~\cite{Hoffman1973} for \textsc{pmc} surfaces in space forms. It is worth pointing out that Hoffman was able to deal with the cases $K \geq 0$ and $K \leq 0$ in both $\mathbb{S}^4$ and $\mathbb{H}^4$, by finding suitable superharmonic functions bounded from below and reducing to the constant Gauss curvature case. On the contrary, Alencar, do Carmo and Tribuzy only treated the case $K \geq 0$ in $\mathbb{H}^n(c) \times \mathbb{R}$. This result has been extended to $\mathbb{S}^n(c)\times\mathbb{R}$ by Fetcu and Rosenberg~\cite[Theorem 1.2]{FR2011} by using a Simon-type equation.

\subsection{PMC surfaces in $\mathbb{C}\mathrm{H}^2$ and $\mathbb{C}\mathrm{P}^2$}
\label{subsec:pmc-surfaces-complex-spaces}

The classification of \textsc{pmc} surfaces in $\mathbb{C}\mathrm{P}^2$ and $\mathbb{C}\mathrm{H}^2$ appeared first in a paper of Kenmotsu and Zhou~\cite{KZ00}. Unfortunately, their result depended upon the structure equations for \textsc{pmc} surfaces given by Ogata~\cite{Ogata95}, which turned out to be incorrect (see~\cite{KO2015} for the correction). However, Hirakawa~\cite[Theorem 2.1]{Hirakawa2006}, who spotted Ogata's mistake, gave a partial solution to the problem, recently completed by Kenmotsu~\cite{Kenmotsu16} in a non-explicit way.

Given an immersion of an oriented surface $\Sigma$ in $\mathbb{C}\mathrm{M}^2(c)$ (or more generally, in any complex manifold), the K{\"a}hler function $C: \Sigma \to [-1,1]$ is defined by $C(p) = \langle{Je_1},{e_2}\rangle$, where $\{e_1,e_2\}$ is an oriented orthonormal basis of $T_p\Sigma$ and $J$ is the complex structure (some authors define $\theta = \arccos C$ as the \emph{K{\"a}hler angle} of the immersion instead). The points $p \in \Sigma$ where $C^2(p) = 1$ are called \emph{complex}, that is, they are the points where $T_p\Sigma$ is complex. Likewise, the points $p$ where $C(p) = 0$ are the points where $T_p\Sigma$ is totally real. In particular, if $C$ is constant zero, then the immersion is Lagrangian.

The main goal of~\cite{Hirakawa2006} was to study \textsc{pmc} surfaces with constant Gauss curvature (in particular, constant K{\"a}hler angle \textsc{pmc} surfaces, see the following paragraph), but Hirakawa also dealt with \textsc{pmc} surfaces satisfying a technical condition in Ogata's equation, namely $a = \bar{a}$, where $a = \langle{J\nabla C},{H + i\widetilde{H}}\rangle$ (see Remark~\ref{rmk:definition-a-complex-space-forms}). This condition implies geometrically the existence of special coordinates in $\Sigma$ such that the $C$ only depends on one coordinate, see~\cite{KO2015}. He also pointed out some examples that were missing in Kenmotsu and Zhou's paper. Hirakawa found, among others, \textsc{pmc} spheres and Delaunay \textsc{cmc} surfaces in $\mathbb{R}^3\subset \mathbb{R}^4$, and gave four different types of solutions in $\mathbb{C}^2$, one type in $\mathbb{C}\mathrm{P}^2$ and $\mathbb{C}\mathrm{H}^2$ with $H^2\geq 2$, and two special types in $\mathbb{C}\mathrm{H}^2$ for $H^2 = 4/3$. Kenmotsu described the rest of \textsc{pmc} examples in $\mathbb{C}\mathrm{P}^2$ and $\mathbb{C}\mathrm{H}^2$, that is, those with $a \neq \bar{a}$ (in particular with non-constant K{\"a}hler angle) in terms of a real-valued harmonic function and five real constants (cp.~\cite[Theorem 5.1]{Hoffman1973}).

\begin{theorem}[{\cite[Theorem 2.1]{Hirakawa2006} and \cite{Kenmotsu16}}]
\label{thm:classification-pmc-complex-space-forms}
Let $\Sigma$ be a \textsc{pmc} surface immersed in $\mathbb{C}\mathrm{M}^2(c)$ and $a: \Sigma\to \mathbb{C}$ given by $a = \langle{J\nabla C},{H + i\widetilde{H}}\rangle$, where $\nabla C$ is the gradient of the K{\"a}hler function and $\widetilde{H}$ is defined in Lemma~\ref{lm:first-properties-pmc}.
\begin{enumerate}
    \item If $a$ is a real-valued function, then one of the following assertions holds:
    \begin{enumerate}
        \item[(i)] $|H|^2 \geq -c/2$ and the immersion is Lagrangian, or
        \item[(ii)] $|H|^2 = -c/3$ and either the K{\"a}hler function is constant $1/3$ or it is a special solution (see item (iii)-2 in~\cite[Theorem 2.1]{Hirakawa2006}).
    \end{enumerate}
    \item If $a \neq \bar{a}$, then the solution depends on one real-valued harmonic function and five real constants.
\end{enumerate}
\end{theorem}

\begin{remark}\label{rmk:definition-a-complex-space-forms}
Our definition of $a$ differs from the definition in \cite{Hirakawa2006, Kenmotsu16} in a multiplicative real function plus a constant term $-\frac12|H|$, which is irrelevant for the statement. Actually,
\[
a = \frac{1}{2|H|(1-C^2)}\langle{\nabla C},{J(H-i\widetilde{H})}\rangle - \frac12|H|,
\]
which is defined in the open dense set $\Sigma\setminus\{p\in \Sigma:\, C(p)^2 = 1\}$ (observe that the interior of the set $\{p\in \Sigma:\, C(p)^2 = 1\}$ is empty since otherwise the interior will be a complex surface, hence minimal, and we are supposing that $\Sigma$ is \textsc{pmc}).
\end{remark}

\begin{remark}
Among the solutions given by Theorem~\ref{thm:classification-pmc-complex-space-forms}, the following are those with constant Gauss curvature (see~\cite[Theorem~1.1]{Hirakawa2006}):
\begin{itemize}
    \item Either $K = -H^2/2$ and $\Sigma \subset \mathbb{C}\mathrm{M}^2(-3|H|^2)$ is (an open piece of):
    \begin{enumerate}
        \item[(i)] the slant surface found by Chen in~\cite{Chen1998}, or
        \item[(ii)] one of the examples described in~\cite[Examples, p.~230]{Hirakawa2006}.
    \end{enumerate}

    \item Or $K = 0$ and the immersion is Lagrangian and $\Sigma$ is (an open piece of):
    \begin{enumerate}
        \item[(i)] the product of two circles in $\mathbb{C}\mathrm{P}^2(c)$, $c > 0$, \cite[Theorem 2]{DT1995}, or
        \item[(ii)] a plane, a cylinder, or a product of two circles in $\mathbb{C}\mathrm{H}^2(c)$, $c < 0$ with $|H|^2 \geq -c/2$, \cite[Theorem 1]{Hirakawa2004}.
    \end{enumerate}
\end{itemize}

In Theorem~\ref{thm:classification-pmc-complex-space-forms} we omitted the case of $\mathbb{C}^2$ on purpose. Nevertheless, Hoffman~\cite[Proposition 3.4]{Hoffman1973} proved that a \textsc{pmc} flat surface in $\mathbb{C}^2$ is part of a cylinder or a product of two circles (see also~\cite[Theorem 7.1]{Chen1990}). Hirakawa also studied \textsc{pmc} surfaces with constant Gauss curvature in $\mathbb{C}^2$ (see items (2)-(b) and (3) in~\cite[Theorem~1.1]{Hirakawa2006} and item (ii) in~\cite[Theorem~2.1]{Hirakawa2006}).
\end{remark}

\subsection[PMC surfaces in Riemannian products]{PMC surfaces in $\mathbb{S}^2\times \mathbb{S}^2$ and $\mathbb{H}^2\times \mathbb{H}^2$}
\label{subsec:pmc-surfaces-Riemannian-products}

The case $\mathbb M^2(\epsilon)\times\mathbb M^2(\epsilon)$, $\epsilon^2 = 1$, is of different nature to the other cases we have presented so far. The classification is still incomplete, being only known under an extra assumption on the \emph{extrinsic normal curvature}. This curvature is defined in the same fashion as the normal curvature $K^\bot$, but using the ambient Riemannian curvature tensor $\overline{R}$ in Equation~\eqref{eqn:normal-curvature} rather than the curvature tensor $R^\bot$.

\begin{theorem}[{\cite[Theorems~2 and~3]{TU12}}]
\label{thm:clasification-pmc-M2xM2}
Let $\phi: \Sigma \rightarrow\mathbb M^2(\epsilon)\times\mathbb M^2(\epsilon)$ be a \textsc{pmc} immersion of an oriented surface $\Sigma$ with vanishing extrinsic normal curvature. Then $\phi$ is locally congruent to
\begin{enumerate}
  \item a \textsc{cmc} surface in a totally geodesic $\mathbb{M}^2(\epsilon)\times \mathbb{M}^1(\epsilon)$, or
  \item a specific example given in~\cite[Example~1 and Proposition~5]{TU12}.
\end{enumerate}
Moreover, if $\phi$ is Lagrangian (not necessarily with vanishing extrinsic normal curvature), then $\phi(\Sigma)$ is an open set of the examples given in~\cite[Example~1]{TU12}.
\end{theorem}

\begin{remark}
Among the examples described in~\cite{TU12}, there are \textsc{pmc} surfaces not lying in a totally geodesic hypersurface of $\mathbb M^2(\varepsilon)\times \mathbb M^2(\varepsilon)$.
\end{remark}

The proof, which will not be sketched here, heavily relies upon the complex structure of $\mathbb M^2(\epsilon)\times\mathbb M^2(\epsilon)$, not only on the product structure as in other cases. It is worth mentioning that there is also a local correspondence between pairs of \textsc{cmc} immersions in $\mathbb M^2(\epsilon)\times \mathbb{R}$ and \textsc{pmc} immersions in $\mathbb{M}^2(\epsilon)\times\mathbb{M}^2(\epsilon)$~\cite[Theorem~1]{TU12}. This relation provides a weak rigidity result for \textsc{cmc} surfaces in $\mathbb{M}^2(\epsilon)\times \mathbb{R}$. It is conjectured that the condition on the extrinsic normal curvature can be dropped, but probably that problem needs a different approach. If this conjecture were true, it would also imply a strong rigidity result for \textsc{cmc} surfaces in $\mathbb{S}^2 \times \mathbb{R}$ and $\mathbb{H}^2 \times \mathbb{R}$ (cf.~\cite[Corollary 3]{TU12}).

\medskip

\textsc{King's College London, Department of Mathematics, Strand wc2r 2ls London, UK.} E-mail address: \texttt{manzanoprego@gmail.com}

\textsc{Centro Universitario de la Defensa. Academia General del Aire. San Javier, Spain.} E-mail address: \texttt{francisco.torralbo@cud.upct.es}

\textsc{KU Leuven, Department of Mathematics, Celestijnenlaan 200B -- Box 2400, 3001 Leuven, Belgium.} E-mail address: \texttt{joeri.vanderveken@kuleuven.be}\\[-20pt]
\begin{center}
\Large{\aldine}
\end{center}
\end{document}